\documentclass[11pt,twoside,reqno]{amsart}
\allowdisplaybreaks
\usepackage{amsmath,amstext,amssymb,epsfig,multicol,enumerate}
\usepackage{graphicx}
\usepackage{mathrsfs}
\calclayout
\makeatletter
\renewcommand{\@seccntformat}[1]{\bf\csname the#1\endcsname.}
\renewcommand{\section}{\@startsection{section}{1}
	\z@{.7\linespacing\@plus\linespacing}{.5\linespacing}
	{\normalfont\upshape\bfseries\centering}}
\renewcommand{\@biblabel}[1]{\@ifnotempty{#1}{#1.}}
\makeatother
\theoremstyle{plain}
\newtheorem{thm}{Theorem}[section]

\newtheorem{prop}[thm]{Proposition}

\theoremstyle{definition}
\newtheorem{ex}[thm]{Example}
\newtheorem{defn}[thm]{Definition}
\newtheorem{rem}{Remark}[section]

\usepackage{cancel}
\usepackage[parfill]{parskip}
\usepackage[german]{varioref}
\usepackage[all]{xy}
\usepackage{color}
\def \>{\succ}
\def \<{\prec}

\begin{document}	
\title[Basdouri Imed\textsuperscript{1}, Bouzid Mosbahi\textsuperscript{2}]{CLASSIFICATION OF ($\rho,\tau,\sigma$)-DERIVATIONS OF TWO-DIMENSIONAL LEFT-SYMMETRIC DIALGEBRAS}
	\author{Basdouri Imed\textsuperscript{1}, Bouzid Mosbahi\textsuperscript{2}}
 \address{\textsuperscript{1}Department of Mathematics, Faculty of Sciences, University of Gafsa, Gafsa, Tunisia}
\address{\textsuperscript{2}Department of Mathematics, Faculty of Sciences, University of Sfax, Sfax, Tunisia}

	%%%%
 \email{\textsuperscript{1}basdourimed@yahoo.fr}
\email{\textsuperscript{2}mosbahi.bouzid.etud@fss.usf.tn}
	
	%%%%%%%%%%%%%%%%%%%%%%%%%%%%%%%%%%%%%%%%%%%%%%%%%%%%%%%%%%%%%%%%%%%%%%%%%%%%%%%
	%%%%%%%%%%%%%%%%%%%%%%%%%%%%%%%%%%%%%%%%%%%%%%%%%%%%%%%%%%%%%%%%%%%%%%%%%%%%%

	
	%[class=AMS]

	%%%%%%%%%%%%%%%%%%%%%%%%%%%%%%%%%%%%%%%%%%%%%%%%%%%%%%%%%
	%%%%%%%%%%%%%%%%%%%%%%%%%%%%%%%%%%%%%%%%%%%%%%%%%%%%%%%%%
	\date{\today}
	%\thanks{This work was supported by}
	%%%%%%%%%%%%%%%%%%%%%%%%%%%%%%%%%%%%%%%%%%%%%%%%%%%%%%%%%%%%%%%%%%%%%%%%%%%%%%%%%%%%%%%%%%%%%%%%%%%%%%%%%%%%%%%%%%

\begin{abstract}
We introduce and study a generalized form of derivations for dendriform algebras, specifying all admissible parameter values that define these derivations. Additionally, we present a complete classification of generalized derivations for two-dimensional left-symmetric dialgebras over the field $\mathbb{K}$.
\end{abstract}
\maketitle
%%%%%%%%%%%%%%%%%%%%%%%%%%%%%%%%%%%%%%%%%%%%%%%
\section{ Introduction}
%%%%%%%%%%%%%%%%%%%%%%%%%%%%%%%%%%%%%%%%%%%%%%%
Leibniz algebras, introduced by J.-L. Loday, are a generalization of Lie algebras. They modify the structure of Lie algebras by removing the requirement that the bracket be skew-symmetric. Instead of the Jacobi identity, Leibniz algebras follow the Leibniz identity, which leads to new and flexible algebraic properties \cite{1}. This change provides a different algebraic framework that connects associative algebras with Lie algebras in a new way, much like Leibniz algebras relate to diassociative algebras (also known as associative dialgebras). Diassociative algebras extend the idea of associative algebras by introducing two different products, which allows them to better represent systems where non-commutativity plays a role. Loday showed that any diassociative algebra can be transformed into a Leibniz algebra by defining a suitable commutator bracket \cite{2,3}.
The concept of left-symmetric dialgebras was further developed by R. Felipe \cite{4}, who explored algebras with two distinct products that include dialgebras as special cases. Left-symmetric dialgebras, formally introduced in 2011, expand on the idea of left-symmetric algebras \cite{5,6,7}.
In this paper, we focus on finding and describing the $(\rho, \tau, \sigma)$-derivations of left-symmetric dialgebras. We provide a step-by-step algorithm for computing these derivations, especially in low-dimensional cases, and present the results in tables. This study builds on the classification results of two-dimensional left-symmetric dialgebras by Rikhsiboev (2016) in \cite{8}.
%%%%%%%%%%%%%%%%%%%%%%%%%%%%%%%%%%%%%%%%%%%%%%%%%%%%%%%%%%%
\section{Preliminaries}
%%%%%%%%%%%%%%%%%%%%%%%%%%%%%%%%%%%%%%%%%%
\begin{defn}
Let  $S$ be a vector space over a field  $\mathbb{K}$  equipped with two associative bilinear binary operations  $\dashv :S \times S \rightarrow S$  and $\vdash:S \times S \rightarrow S$ satisfying the axioms:
\begin{align*}
p \dashv (q \dashv r) &= p \dashv (q \vdash r),\\
(p \vdash q) \dashv r &= p \vdash (q \dashv r),\\
(p \vdash q) \vdash r &= (p \dashv q) \vdash r,
\end{align*}
for all $p, q, r \in S$. Then the triple  $(S, \dashv, \vdash)$  is called a \textit{diassociative algebra}.
\end{defn}

\begin{defn}
Let  $S$ be a vector space over a field $\mathbb{K}$. Suppose $S$ is equipped with two bilinear products, not necessarily associative,
\begin{align*}
&\dashv : S \times S \rightarrow S \quad \text{and} \quad \vdash : S \times S \rightarrow S,
\end{align*}
satisfying the following identities for all  $p, q, r \in S$:
\begin{align*}
    p \dashv (q \dashv r) &= p \dashv (q \vdash r), \\
    (p \vdash q) \vdash r &= (p \dashv q) \vdash r, \\
    p \dashv (q \dashv r) - (p \dashv q) \dashv r &= q \vdash (p \dashv r) - (q \vdash p) \dashv r, \\
    p \vdash (q \vdash r) - (p \vdash q) \vdash r &= q \vdash (p \vdash r) - (q \vdash p) \vdash r,
\end{align*}
Then  $S$  is called a \textit{left-symmetric dialgebra} or a \textit{left disymmetric algebra}.
\end{defn}
The operations $\dashv$ and $\vdash$ are called left and right products, respectively.
%%%%%%%%%%%%%%%%%%%%%%%%%%%%%%%%%%%%%%%%%%02

\begin{defn}
A \textit{morphism} of left-symmetric dialgebras from  $S$  to $S'$ is a linear map $\pi: S \to S'$ such that
\begin{align*}
\pi(p \dashv q) &= \pi(p) \dashv \pi(q) \quad \text{and} \quad \pi(p \vdash q) = \pi(p) \vdash \pi(q),
\end{align*}
for all $p, q \in S$.
\end{defn}

\begin{rem}
A bijective homomorphism is an isomorphism of $S$  and  $S'$.
\end{rem}

\begin{ex}
Any left-symmetric algebra  $S$ is a left-symmetric dialgebra since we can define two products  $\dashv$ and $\vdash$ on  $S$ as follows: for any  $p, q \in S$,
\begin{align*}
p \dashv q &= pq = p \vdash q,
\end{align*}
where $pq$ denotes the product in the associative algebra $S$.
\end{ex}
\begin{ex}
Let  $\mathbb{K}[p, q]$ be a polynomial algebra over a field $\mathbb{K}\; (char\mathbb{K} = 0)$,
with two indeterminates  $p$ and  $q$. We define the left product $\dashv$ and the right product $\vdash$ on $\mathbb{K}[p, q]$ as follows:

\begin{align*}
\phi(p, q)\dashv \psi(p, q) = \phi(p,q)\dashv \psi(q,q)
\end{align*}
and
\begin{align*}
\phi(p,q)\vdash \psi(p,q) = \phi(p,q) \vdash \psi(p,q).
\end{align*}

Then  $(\mathbb{K}[p, q], \dashv,\vdash)$ is a left-symmetric dialgebra.
\end{ex}

\begin{defn}
Let $S$ be a left-symmetric dialgebra. A \textit{derivation} of  $S$  is a linear transformation  $d : S \rightarrow S$  satisfying
\begin{align*}
d(p \dashv q) &= d(p) \dashv q + p \dashv d(q) \quad \text{and} \quad d(p \vdash q) = d(p) \vdash q + p \vdash d(q)
\end{align*}
for all $p, q \in S$.
\end{defn}

\begin{defn}
A left-symmetric dialgebra $S$ is called \textit{characteristically nilpotent} if $Der(S)$ (the set of all derivations) is nilpotent as a Lie algebra.
\end{defn}
%%%%%%%%%%%%%%%%%%%%%%%%%%%%%%%%%%%%%%%%%%%%%%%
\section{ $(\rho, \tau, \sigma)$-Derivations of Left-Symmetric Dialgebras}
%%%%%%%%%%%%%%%%%%%%%%%%%%%%%%%%%%%%%%%%%%%%%%%

\begin{defn}
A linear operator $d \in \text{End}_S$ is said to be an $(\rho, \tau, \sigma)$-derivation of a left symmetric dialgebra $S$, where $(\rho, \tau, \sigma)$ are fixed elements of a field $\mathbb{K}$, if for all $p, q \in S$:

\begin{align*}
\rho d(p \dashv q) &= \tau d(p) \dashv q + \sigma p \dashv d(q)
\end{align*}
and
\begin{align*}
\rho d(p \vdash q) &= \tau d(p) \vdash q + \sigma p \vdash d(q),
\end{align*}

where $\dashv$ and $\vdash$ denote the operations of the left symmetric dialgebra $S$.
\end{defn}
The set of all $(\rho, \tau, \sigma)$-derivations of $S$ is denoted by $\text{Der}_{(\rho, \tau, \sigma)}S$. It is evident that $\text{Der}_{(\rho, \tau, \sigma)}S$ is a linear subspace of $\text{End}_S$.
\begin{prop}\label{P1}
Let $S$ be a complex left symmetric dialgebra. Then the values of $\rho, \tau, \sigma$ in  $Der_{(\rho,\tau,\sigma)}S$ are distributed as follows:
\begin{align*}
\text{Der}_{(1,1,1)}S & = \{ d \in \text{End}S \mid d(p \ast q) = d(p) \ast q + p \ast d(q) \}, \\
\text{Der}_{(1,1,0)}S & = \{ d \in \text{End}S \mid d(p \ast q) = d(p) \ast q \}, \\
\text{Der}_{(1,0,1)}S & = \{ d \in \text{End}S \mid d(p \ast q) = p \ast d(q) \}, \\
\text{Der}_{(1,0,0)}S & = \{ d \in \text{End}S \mid d(p \ast q) = 0 \}, \\
\text{Der}_{(0,1,1)}S & = \{ d \in \text{End}S \mid d(p) \ast q = p \ast d(q) \}, \\
\text{Der}_{(0,0,1)}S & = \{ d \in \text{End}S \mid p \ast d(q) = 0 \}, \\
\text{Der}_{(0,1,0)}S & = \{ d \in \text{End}S \mid d(p) \ast q = 0 \}, \\
\text{Der}_{(0,1,\delta)}S & = \{ d \in \text{End}S \mid d(p) \ast q = \delta p \ast d(q), \delta \in \mathbb{C} \setminus \{0, 1\} \},
\end{align*}
where $\ast=\dashv,\vdash$ denotes the operation in the left symmetric dialgebra.
\end{prop}

\begin{proof}
Let $\rho \neq 0$. Applying the operator  $d$  to the identities of left symmetric dialgebras, we obtain the system of equations:
\begin{align*}
\tau(\tau - \rho) &= 0 \quad \text{and} \quad \sigma(\sigma - \rho) = 0.
\end{align*}
By considering each case, we find the following possible values of
\begin{align*}
&(\rho, \tau, \sigma) :  (\rho, \rho, \rho), (\rho, \rho, 0), (\rho, 0, \rho),\text{and }(\rho, 0, 0).
\end{align*}
Taking into account the fact that
\begin{align*}
\text{Der}_{(\rho,\tau,\sigma)}S &= \text{Der}_{(1,\frac{\tau}{\rho},\sigma)}S,
\end{align*}
we obtain the first four equalities in Proposition \ref{P1}
Now let $\rho = 0$. Then we have:
$$
(0, \tau, \tau), (0, \tau, 0)\;  \text{if}\;  \tau \neq 0,\; \text{and}\;
(0, 0, \sigma)\;  \text{if}\; \tau = 0.
$$
Thus, we find:
$$
\text{Der}_{(0,\rho,\sigma)}S = \text{Der}_{(0,1,\frac{\sigma}{\tau})}S \quad \text{if } \tau \neq 0,
$$
and
$$
\text{Der}_{(0,0,\sigma)}S = \text{Der}_{(0,0,1)}S \quad \text{if } \tau = 0 \text{ and } \sigma \neq 0.
$$
Finally, if $\sigma = 0$, then $\text{Der}_{(0,0,0)}S = \text{End}S$.
\end{proof}

\begin{prop}
Let $S$ be a left symmetric dialgebra, and let $d_1, d_2$ be $(\rho, \tau, \sigma)$-derivations on $S$. Then $[d_1, d_2] = d_1 \circ d_2 - d_2 \circ d_1$ is a ($\rho^2, \tau^2, \sigma^2$)-derivation of  $S$.
\end{prop}

\begin{proof}
Let us consider the following two equations:
\begin{align}\label{E1}
\rho d_1(p \ast q) &= \tau d_1(p) \ast q + \sigma p \ast d_1(q)
\end{align}
and
\begin{align}\label{E2}
\rho d_2(p \ast q) &= \tau d_2(p) \ast q + \sigma p \ast d_2(q).
\end{align}

\textbf{Case 1:} $\rho \neq 0$.

\begin{align*}
\rho^2 [d_1, d_2](p \ast q) &= \rho^2(d_1 \circ d_2)(p \ast q) - \rho^2(d_2 \circ d_1)(p \ast q)\\
&= \rho d_1(\rho d_2(p \ast q)) - \rho d_2(\rho d_1(p \ast q))\\
&= \rho d_1(\tau d_2(p) \ast q + \sigma p \ast d_2(q)) - \rho d_2(\tau d_1(p) \ast q + \sigma p \ast d_1(q))\\
&= \tau (\rho d_1(d_2(p) \ast q)) + \sigma (\rho d_1(p \ast d_2(q))) - \tau (\rho d_2(d_1(p) \ast q) + \sigma \rho d_2(p \ast d_1(q))).
\end{align*}

Finally, we get
\begin{align*}
\rho^2 [d_1, d_2](p \ast q) &= \tau^2 [d_1, d_2](p) \ast q + \sigma^2 p \ast [d_1, d_2](q).
\end{align*}

\textbf{Case 2:} $\rho = 0$. Then from equations (\ref{E1}) and (\ref{E2}), we have:
\begin{align}\label{E3}
\tau d_1(p) \ast q &= -\sigma p \ast d_1(q)
\end{align}
and
\begin{align}\label{E4}
\tau d_2(p) \ast q &= -\sigma p \ast d_2(q).
\end{align}

Therefore,
\begin{align*}
\tau^2 [d_1, d_2](p \ast q) &= \tau^2 (d_1 \circ d_2)(p) \ast q - \tau^2 (d_2 \circ d_1)(p) \ast q\\
&= \tau^2 (d_1 d_2)(p) \ast q - \tau^2 (d_2 d_1)(p) \ast q\\
&= \tau \left( \tau d_1(d_2(p) \ast q) - \tau d_2(d_1(p) \ast q) \right).
\end{align*}
Using equations (\ref{E3}) and (\ref{E4}), we obtain:
\begin{align*}
\tau^2 [d_1, d_2](p) \ast q &= \tau(-\sigma d_2(p) d_1(q)) + \tau(\sigma d_1(p) d_2(q))\\
&= -\sigma(\tau d_2(p) d_1(q) + \sigma(\tau d_1(p) d_2(q))).
\end{align*}
Due to (\ref{E3}) and (\ref{E4}) again, we have:
\begin{align*}
\sigma^2 p \ast d_2(d_1(q)) - \sigma^2 p \ast d_1(d_2(q)) &= \sigma^2 p \left[ d_2(d_1(q)) - d_1(d_2(q)) \right]= \sigma^2 p \ast [d_1, d_2](q).
\end{align*}
\end{proof}

\begin{rem}
For any $\rho, \tau, \sigma \in \mathbb{C}$, the dimension of the vector space
\(\text{Der}_{(\alpha, \beta, \gamma)} S\) is an isomorphism invariant of left symmetric dialgebras.
\end{rem}

%%%%%%%%%%%%%%%%%%%%%%%%%%%%%%%%%%%%%%%%%%%%%%%%%%%%%%%%%%%
\section{ $(\rho, \tau, \sigma)$-Derivations of Low-Dimensional Left-Symmetric Dialgebras}
%%%%%%%%%%%%%%%%%%%%%%%%%%%%%%%%%%%%%%%%%%
This section is devoted to the description of generalized derivations of two-dimensional left-symmetric dialgebras. Let $\{e_1, e_2, e_3, \ldots, e_n\}$ be a basis of an $n$-dimensional left-symmetric dialgebra $S$. Then the multiplication in the algebra can be expressed as follows:
\begin{align*}
e_i\dashv e_j &= \sum_{k=1}^{n} \gamma_{ij}^{k} e_k \quad \text{and} \quad e_s\dashv e_t = \sum_{l=1}^{n} \gamma_{st}^{l} e_l,
\end{align*}
Where $i, j, s, t = 1, 2, \dots, n$. The coefficients of the above linear combinations $\{\gamma^k_{ij}, \delta^l_{st}\} \in \mathbb{K}^{2n^3}$ are called the structure constants of $S$ on the basis $\{e_1, e_2, e_3, \dots, e_n\}$. A generalized derivation being a linear transformation of the vector space $S$ is represented in a matrix form $[d_{ij}]_{i,j=1,2,\dots,n}$, i.e.
\begin{align*}
d(e_i) &= \sum_{j=1}^n d_{ij} e_j, \quad i = 1, 2, \dots, n.
\end{align*}
According to the definition of the generalized derivation, the entries $d_{ij}, i, j = 1, 2, \dots, n$, of the matrix $[d_{ij}]_{i,j=1,2,\dots,n}$ must satisfy the following system of equations:
\begin{align*}
\sum_{t=1}^n (\rho \gamma^t_{ij} d_{st} - \tau d_{ti} \gamma^s_{tj} - \sigma d_{tj} \gamma^s_{it}) &= 0, \quad \text{for } i, j, s = 1, 2, 3, \dots, n,\\
\sum_{l=1}^n (\rho \delta^l_{st} d_{ml} - \tau d_{ls} \delta^m_{lt} - \sigma d_{lt} \delta^m_{sl}) &= 0, \quad \text{for } s, t, m = 1, 2, 3, \dots, n.
\end{align*}
It is observed that if the structure constants $\{\gamma^k_{ij}, \delta^l_{st}\}$ of of a left-symmetric dialgebra $S$ are given, then in order to describe its generalized derivation one has to solve the system of equations above with respect to $d_{ij}, i, j = 1, 2, \dots, n$, which can be done by using computer software. The values of structure constants $\gamma^k_{ij}$ and $\delta^l_{st}$ for two-dimensional left-symmetric dialgebra are obtained from the classification results of Rikhsiboev (2016), where  $\text{LSD}_n^m$ denotes the $m-th$ isomorphism class of $n$-dimensional complex left-symmetric dialgebras.
\begin{thm}
Any two-dimensional left-symmetric dialgebra $S$ is included in one of the following isomorphism classes of algebras:
For all $a,b,c \in \mathbb{K}$\\
$L^1_2(a,b):e_1 \dashv e_2 = e_1, \quad e_2 \dashv e_2 = e_2, \quad e_2 \vdash e_1 = a e_1, \quad e_2 \vdash e_2 = b e_1 + e_2, \quad a \neq 0.$\\
$L^2_2(b,c):e_1 \dashv e_2 = e_1, \quad e_2 \dashv e_2 = c e_1 + e_2, \quad e_2 \vdash e_2 = b e_1 + e_2, \quad c \neq 0.$\\
$L^3_2(b):e_1 \dashv e_2 = e_1, \quad e_2 \dashv e_2 = e_2, \quad e_2 \vdash e_2 = b e_1 + e_2.$\\
$L^4_2(c):e_2 \dashv e_2 = c e_1 + e_2, \quad e_2 \vdash e_1 = e_1, \quad e_2 \vdash e_2 = e_2.$\\
$L^5_2(a,c):e_2 \dashv e_2 = c e_1 + e_2, \quad e_2 \vdash e_1 = a e_1, \quad e_2 \vdash e_2 = c(1 - a)e_1 + e_2, \quad a \neq 1.$\\
$L^6_2(a):e_2 \dashv e_2 = e_2, \quad e_2 \vdash e_1 = a e_1, \quad e_2 \vdash e_2 = e_2, \quad a \neq 0.$
\end{thm}
The $(\rho, \tau, \sigma)$-derivations of the study of two-dimensional left-symmetric dialgebras over a field $\mathbb{K}$ are given in the following tables.
\vspace{3in}
\begin{center}
\begin{tabular}{|c|c|c|c|}
\hline
\textbf{IC} & \textbf{($\rho$, $\tau$, $\sigma$)} & \textbf{${\text{Der}}_{(\rho, \tau, \sigma)} S$} & \textbf{Dim} \\
\hline
$L^1_2(a,b)$ & (1,1,1) & $\left(\begin{array}{cc} \frac{d_{12}(a-1)}{b} & d_{12} \\ 0 & 0 \end{array}\right)$ & 1 \\
\cline{2-4}
& (1,1,0) & $\left(\begin{array}{cc} d_{22} & 0 \\ 0 & d_{22} \end{array}\right)$ & 1 \\
\cline{2-4}
& (1,0,1) & $\left(\begin{array}{cc} d_{22} & 0 \\ 0 & d_{22} \end{array}\right)$ & 1 \\
\cline{2-4}
& (1,0,0) & trivial & 0 \\
\cline{2-4}
& (0,1,1) & trivial & 0 \\
\cline{2-4}
& (0,0,1) & trivial & 0 \\
\cline{2-4}
& (0,1,0) & trivial & 0 \\
\cline{2-4}
& (0,1,\(\delta\)) & $\left(\begin{array}{cc} d_{22} & 0 \\ 0 & d_{22} \end{array}\right)$ & 1 \\
\hline

$L^2_2(b,c)$ & (1,1,1) & $\left(\begin{array}{cc} 0 & 0 \\ 0 & 0 \end{array}\right)$ & 0 \\
\cline{2-4}
& (1,1,0) & $\left(\begin{array}{cc} d_{22} & 0 \\ 0 & d_{22} \end{array}\right)$ & 1 \\
\cline{2-4}
& (1,0,1) & $\left(\begin{array}{cc} d_{22} & 0 \\ 0 & d_{22} \end{array}\right)$ & 1 \\
\cline{2-4}
& (1,0,0) & trivial & 0 \\
\cline{2-4}
& (0,1,1) & trivial & 0 \\
\cline{2-4}
& (0,0,1) &$\left(\begin{array}{cc} d_{11} & d_{12} \\ 0 & 0 \end{array}\right)$ & 2\\
\cline{2-4}
& (0,1,0) & trivial & 0 \\
\cline{2-4}
& (0,1,\(\delta\)) & $\left(\begin{array}{cc} d_{22} & 0 \\ 0 & d_{22} \end{array}\right)$ & 1 \\
\hline
$L^3_2(b)$ & (1,1,1) & $\left(\begin{array}{cc}d_{11} & -bd_{11} \\ 0 & 0 \end{array}\right)$ & 1 \\
\cline{2-4}
& (1,1,0) & $\left(\begin{array}{cc} d_{11} & -bd_{11}+bd_{22} \\ 0 & d_{22} \end{array}\right)$ & 2 \\
\cline{2-4}
& (1,0,1) & $\left(\begin{array}{cc} d_{22} & 0 \\ 0 & d_{22} \end{array}\right)$ & 1 \\
\cline{2-4}
& (1,0,0) & trivial & 0 \\
\cline{2-4}
& (0,1,1) & trivial & 0 \\
\cline{2-4}
& (0,0,1) & $\left(\begin{array}{cc} d_{11} & d_{12} \\ 0 & 0 \end{array}\right)$ & 2 \\
\cline{2-4}
& (0,1,0) & trivial & 0 \\
\cline{2-4}
& (0,1,\(\delta\)) & $\left(\begin{array}{cc} d_{22} & 0 \\ 0 & d_{22} \end{array}\right)$ & 1 \\
\hline
\end{tabular}
\end{center}
\begin{center}
\begin{tabular}{|c|c|c|c|}
\hline
\textbf{IC} & \textbf{($\rho$, $\tau$, $\sigma$)} & \textbf{${\text{Der}}_{(\rho, \tau, \sigma)} S$} & \textbf{Dim} \\
\hline
$L^4_2(c)$ & (1,1,1) & $\left(\begin{array}{cc}d_{11}& -cd_{11} \\ 0 & 0 \end{array}\right)$ & 1 \\
\cline{2-4}
& (1,1,0) & $\left(\begin{array}{cc} d_{11} & -bd_{11}+bd_{22} \\ 0 & d_{22} \end{array}\right)$ & 2 \\
\cline{2-4}
& (1,0,1) & $\left(\begin{array}{cc} d_{11} & -cd_{11}+cd_{22} \\ 0 & d_{22} \end{array}\right)$ & 2 \\
\cline{2-4}
& (1,0,0) & trivial & 0 \\
\cline{2-4}
& (0,1,1) & trivial & 0 \\
\cline{2-4}
& (0,0,1) & trivial & 0 \\
\cline{2-4}
& (0,1,0) &$\left(\begin{array}{cc} d_{11} & d_{12} \\ 0 & 0 \end{array}\right)$ & 2 \\
\cline{2-4}
& (0,1,\(\delta\)) & $\left(\begin{array}{cc} d_{11} & 0 \\ 0 & d_{11} \end{array}\right)$ & 1 \\
\hline
$L^5_2(a,c)$ & (1,1,1) & $\left(\begin{array}{cc}d_{11} & -cd_{11} \\ 0 & 0 \end{array}\right)$ & 1 \\
\cline{2-4}
& (1,1,0) & $\left(\begin{array}{cc} d_{22} & 0 \\ 0 & d_{22} \end{array}\right)$ & 1 \\
\cline{2-4}
& (1,0,1) &  $\left(\begin{array}{cc} d_{11} & -cd_{11}+cd_{22} \\ 0 & d_{22} \end{array}\right)$ & 2 \\
\cline{2-4}
& (1,0,0) & trivial & 0 \\
\cline{2-4}
& (0,1,1) & trivial & 0 \\
\cline{2-4}
& (0,0,1) & trivial & 0 \\
\cline{2-4}
& (0,1,0) & $\left(\begin{array}{cc} d_{11} & d_{12} \\ 0 & 0 \end{array}\right)$ & 2 \\
\cline{2-4}
& (0,1,\(\delta\)) & $\left(\begin{array}{cc} d_{22} & 0 \\ 0 & d_{22} \end{array}\right)$ & 1 \\
\hline

$L^6_2(a)$ & (1,1,1) & $\left(\begin{array}{cc} d_{11}& 0 \\ 0 & 0 \end{array}\right)$ & 1 \\
\cline{2-4}
& (1,1,0) & $\left(\begin{array}{cc} d_{22} & 0 \\ 0 & d_{22} \end{array}\right)$ & 1 \\
\cline{2-4}
& (1,0,1) & $\left(\begin{array}{cc} d_{11} & 0 \\ 0 & d_{22} \end{array}\right)$ & 2 \\
\cline{2-4}
& (1,0,0) & trivial & 0 \\
\cline{2-4}
& (0,1,1) & trivial & 0 \\
\cline{2-4}
& (0,0,1) & trivial & 0 \\
\cline{2-4}
& (0,1,0) & $\left(\begin{array}{cc} d_{11} & d_{12} \\ 0 & 0 \end{array}\right)$ & 2 \\
\cline{2-4}
& (0,1,\(\delta\)) & $\left(\begin{array}{cc} d_{11} & d_{12} \\ 0 & 0 \end{array}\right)$ & 2 \\
\hline
\end{tabular}
\end{center}

\newpage

\textbf{Conclusion:}
This study enabled us to calculate the generalized derivations of two-dimensional left-symmetric dialgebras and to determine the dimension of the generalized derivation space for each representative class, ranging from 0 to 2.\\
\textbf{Conflict of Interests:}
The author declares that there are no conflicts of interest.

\textbf{Acknowledgment:}
We would like to express our appreciation to the referees for their valuable comments and suggestions.

\end{document}